\documentclass[reqno,11pt]{amsart}
\usepackage{amssymb}
\usepackage{amsmath}
\usepackage{mathrsfs}
\usepackage[usenames]{color}
\usepackage{bm}
\usepackage{setspace}
\numberwithin{equation}{section}
\allowdisplaybreaks

\begin{document}
%\renewcommand{\footnote}{}
%\pagenumbering{roman}
%\thispagestyle{empty}%\quad\newpage
%\thispagestyle{empty}
%\nonstopmode
%**************************************************************************

\newtheorem{theorem}{Theorem}[section] % Nummerierung
\newtheorem{proposition}[theorem]{Proposition}
\newtheorem{corollary}[theorem]{Corollary}
\newtheorem{lemma}[theorem]{Lemma}

\theoremstyle{definition}
\newtheorem{assumption}[theorem]{Assumption}
\newtheorem{definition}[theorem]{Definition}

\theoremstyle{definition} %%{remark}
\newtheorem{remark}[theorem]{Remark}
\newtheorem{remarks}[theorem]{Remarks}
\newtheorem{example}[theorem]{Example}
\newtheorem{examples}[theorem]{Examples}
%**************************************************************************
\newenvironment{pf}%
{\begin{sloppypar}\noindent{\bf Proof.}}%
{\hspace*{\fill}$\square$\vspace{6mm}\end{sloppypar}}
%**************************************************************************
\def\bA{{\bm A}}
\def\bB{{\bm B}}
\def\mE{{\mathbb E}}
\def\mK{{\mathbb K}}
\def\hmE{{\widehat{\mathbb E}}}
\def\mEp{{\mathbb E}_{\phi}}
\def\mFp{{\mathbb F}_{\phi}}
\def\mEpp{{\mathbb E}_{\phi,\mP}}
\def\tmEp{\widetilde{\mathbb E}_{\phi}}
\def\tmEpp{\widetilde{\mathbb E}_{\phi,P}}
\def\mFpp{{\mathbb F}_{\phi,\mP}}
\def\tPhi{\widetilde{\Phi}}
\def\mF{{\mathbb F}}
\def\mG{{\mathbb G}}
\def\mX{{\mathbb X}}
\def\mP{{\mathbb P}}
\def\db{\|}
\def\r{Nr}
\def\R{{\mathbb R}}
\def\N{{\mathbb N}}
\def\C{{\mathbb C}}
\def\Q{{\mathbb Q}}
\def\mP{{\mathbb P}}
\def\Z{{\mathbb Z}}
\def\mH{\mathbb H}
\def\mA{\mathbb A}
\def\mT{\mathbb T}
\def\D{{\mathcal D}}
\def\cB{{\mathcal B}}
\def\E{{\mathcal E}}
\def\cF{{\mathcal F}}
\def\cA{{\mathcal A}}
\def\cH{{\mathcal H}}
\def\G{{\mathcal G}}
\def\B{{\mathcal B}}
\def\I{{\mathcal I}}
\def\M{{\mathcal M}}
\def\O{{\mathcal O}}
\def\S{{\mathcal S}}
\def\cT{{\mathcal T}}
\def\cP{{\mathcal P}}
\def\L{{\mathcal L}}
\def\cK{{\mathcal K}}
\def\cJ{{\mathcal J}}
\def\cS{{\mathcal S}}
\def\bH{{\bf H}}
\def\bP{{\bf P}}
\def\bQ{{\bf Q}}
\def\bE{{\bf E}}
\def\bT{{\bf T}}
\def\W{W}
\def\Be{L_\infty}
\def\cR{{\mathcal R}}
\def\eps{\varepsilon}
\def\3{{\ss}}
\def\slim{s-\lim_}
\def\capa{{\mathrm{Cap}}}
\def\supp{{\mathrm{supp}}}
\def\esssup{{\mathrm{ess\,sup}}}
\def\absconv{{\mathrm{absconv}}}
\def\dom{{\mathrm{dom}}}
\def\loc{\mathrm{loc}}
\def\hs{half-space}
\def\HIC{$\HH^\infty$-calculus}
\def\BIP{{\mathrm{BIP}}}                                                 
\def\BUC{{\mathrm{BUC}}}
\def\BC{{\mathrm{BC}}}
\def\MR{{\mathcal{MR}}}
\def\const{{\mathrm{const\,}}}
\def\Re{{\mathrm{Re}}}
\def\re{{\mathrm{Re}}}
\def\Im{{\mathrm{Im}}}
\def\im{{\mathrm{Im}}}
\def\dd{{\mathrm d}}
\def\e{{\mathrm{e}}}
\def\id{{\mathrm{id}}}
\def\sb{{\mathrm{sb}}}
\def\FM{{\mathrm{FM}}}
\def\ME{{\mathrm{M}}}
\def\hperp{{^{_\perp}}}
\def\HIC{{$H^\infty$-calculus}}
\def\hW{\widehat{W}}
\def\hu{\hat{u}}
\def\hv{\hat{v}}
\def\hw{\hat{w}}
\def\hsigma{\hat{\sigma}}
\def\hf{\hat{f}}
\def\hh{\hat{h}}
\def\hg{\hat{g}}
\def\dR{\dot{\R}}
\def\tu{\tilde{u}}
\def\tc{\tilde{c}}
\def\tp{\tilde{p}}
\def\tf{\tilde{f}}
\def\th{\tilde{h}}
\def\tg{\tilde{g}}
\def\tv{\tilde{v}}
\def\ta{\tilde{a}}
\def\ty{\widetilde{y}}
\def\bv{\bar{v}}
\def\bw{\bar{w}}
\def\tsigma{\tilde{\sigma}}
\def\hphi{\hat{\phi}}
\newcommand{\essinf}[1]{{\mathrm{ess}}\!\inf_{\!\!\!\!\!\!\!\!\! #1}}
\newcommand{\fn}{\footnote}
\def\mdt{mixed derivative theorem}
\def\div{{\mathrm {div\,}}}
\def\bsigma{\bar{\sigma}}
\def\brho{\bar{\rho}}
\def\rcv{\W^1_p(J;L^p(\R^{n+1}_+))
        \cap L_p(J;\W^2_p(\R^{n+1}_+))}
\def\rcvtp{\W^1_p(J;L^p(\dR^{n+1}))
        \cap L_p(J;\W^2_p(\dR^{n+1}))}
\def\rcs{W^{3/2-1/2p}_p(J;L_p(\R^n))
        \cap \W^1_p(J;W^{1-1/p}_p(\R^n))
        \cap L_p(J;W^{2-1/p}_p(\R^n))}
\def\rcf{L_p(J;L_p(\R^{n+1}_+))}
\def\rcftp{L_p(J;L_p(\dR^{n+1}))}
\def\rch{W^{1/2-1/2p}_p(J;L_p(\R^n))
        \cap L_p(J;W^{1-1/p}_p(\R^n))}
\def\rcvi{W^{2-2/p}_p(\R^{n+1}_+)}
\def\rcvitp{W^{2-2/p}_p(\dR^{n+1})}
\def\rcsi{W^{2-2/p}_p(\R^n)}
\newcommand{\ab}{&\hskip-2mm}

\def\en{{\talloblong}}
\def\hookd{\stackrel{_d}{\hookrightarrow}}
\def\hook{{\hookrightarrow}}
\def\vt{{\vartheta}}
\def\ovt{{\overline{\vartheta}}}
\def\THE{tornado-hurricane equations}
\def\THO{tornado-hurricane operator}
\def\NSE{Navier-Stokes equations}
\def\SO{Stokes operator}
\def\HHP{Helmholtz projection}
\def\HHD{Helmholtz decomposition}
\def\HOL{\mathrm{HOL}}
\def\la{{\langle}}
\def\ra{{\rangle}}
\def\vphi{{\varphi}}
\def\vdp{{\{k:\ \alpha^v_k=0\}}}
\def\vndp{{\{k:\ \alpha^v_k\neq0\}}}
\def\tdp{{\{k:\ \alpha^\vt_k=0\}}}
\def\tndp{{\{k:\ \alpha^\vt_k\neq0\}}}
\def\sD{{\mathscr D}}
\def\hsD{{\widehat{\mathscr D}}}
\def\sL{{\mathscr L}}
\def\sR{{\mathscr R}}
\def\sT{{\mathscr T}}
\def\sLis{{\mathscr L}_{is}}
\def\PPr{{\PP_{\!\rho}}}
\def\ou{{\overline{u}}}
\def\oq{{\overline{q}}}
\def\oU{{\overline{U}}}
\def\th{{T\!H}}
\def\ttau{{\tilde \tau}}
\def\hH{{\widehat{H}}}
\def\cD{{\mathcal D}}
\def\vp{{\varphi}}
\def\Hic{{\mathcal{H}^\infty}}

\hyphenation{Lipschitz}

%\newcommand\fn[1]{}

%**************************************************************************
\sloppy
%**************************************************************************
\title[Turbulence in active fluids]
{Turbulence in active fluids caused by self-propulsion}

\author[C. Bui]{Christiane Bui}
\address{Mathematisches Institut, Angewandte Analysis\\
         Heinrich-Heine-Uni\-ver\-sit\"at D\"usseldorf\\
	          40204 D\"usseldorf, Germany}
		  \email{Christiane.Bui@uni-duesseldorf.de}

		  \author[H. L\"owen]{Hartmut L\"owen}
		  \address{Institut f\"ur Theoretische Physik II - Soft
		  Matter\\
		           Heinrich-Heine-Uni\-ver\-sit\"at
			   D\"usseldorf\\
			            Universit\"atsstra{\ss}e 1\\
				             40225 D\"usseldorf,
					     Germany}
					     \email{hlowen@thphy.uni-duesseldorf.de}

					     \author[J. Saal]{J\"urgen
					     Saal}
					     \address{Mathematisches
					     Institut, Angewandte
					     Analysis\\
					              Heinrich-Heine-Uni\-ver\-sit\"at
						      D\"usseldorf\\
						               40204
							       D\"usseldorf,
							       Germany}
							       \email{juergen.saal@hhu.de}

%\date{\today}
%\thanks{class-stefan75.tex}
%\thispagestyle{empty}
%\pagestyle{myheadings}
%\markboth{\today}{\today}
%\setlength{\parindent}{0mm}
%\setlength{\parskip}{10cm}
\parskip0.5ex plus 0.5ex minus 0.5ex
%\bibliographystyle{plain}
%%\bibliographystyle{alpha}
%\bibliographystyle{amsalpha}
%\setcounter{page}{3}
%\listoffigures

\begin{abstract}
A rigoros analytical justification of turbulence 
observed in active fluids and caused by self-propulsion is presented.
We prove exis\-tence of unstable wave modes for the generalized
Stokes and Navier-Stokes systems by developing an approach in
spaces of Fourier transformed Radon measures.
\end{abstract}
\maketitle
%%%%%%%%%%%%%%%%%%%%%%%%%%%%%%%%%%%%

 
%\tableofcontents

%\doublespacing

%%%%%%%%%%%%%%%%%%%%%%%%%%%%%%%%%%%%%%%%%%%%%%%%%%%%%%%%%%%%%%%%%%%%%%%%%%%%%
\section{Introduction}
%%%%%%%%%%%%%%%%%%%%%%%%%%%%%%%%%%%%%%%%%%%%%%%%%%%%%%%%%%%%%%%%%%%%%%%%%%%%%
%%%%%%%%%%%%%%%%%%%%%%%%%%%%%%%%%%%%%%%%%%%%%%%%%%%%%%%%%%%%%%%%%%%%%%%%%%%%%

In this brief note we study analytical properties of the 
following minimal hydrodynamic model to describe the bacterial velocity in
the case of highly concentrated bacterial suspensions with negligible density
fluctuations considered on the domain $(0,\infty)\times\R^n$:
\begin{equation}
	\label{eqn:min-hyd-mod}
	\begin{array}{rl}
		v_t+\lambda_0v\cdot\nabla v & =  f-\nabla
		p+\lambda_1\nabla|v|^2-(\alpha+\beta|v|^2)v+\Gamma_0\Delta
		v-\Gamma_2 \Delta^2v,\\
		\mbox{div}\,v & =  0,\\
		v(0)&=v_0.
	\end{array}
\end{equation}
Here $v$ is the bacterial velocity field and
$p$ the (scalar) pressure.
For $\lambda_0=1$, $\lambda_1 = \alpha=\beta = \Gamma_2 = 0$ and $\Gamma_0 >0$, 
the model reduces to the incompressible Navier-Stokes equations in $n$
spatial dimensions. For non-vanishing $\lambda_1,\alpha,\beta,\Gamma_2$ system
(\ref{eqn:min-hyd-mod}) serves as a model to describe 
occuring turbulence in low Reynolds regimes caused by self-propulsion.
It was originally proposed by 
Wensink et al.\ in \cite{Wensink-et-al:Meso-scale-turbulence} 
and then considered in Refs.\ \cite{Dunkel_Heidenreich_PRL_2013,Dunkel_NJP_2013}
and is by now one of the standard models to describe active turbulence at low Reynolds number \cite{Oza}.
The model was recently derived from more microscopic descriptions \cite{Klapp} and was quantitatively confirmed in
suspensions of living biological systems \cite{Wensink-et-al:Meso-scale-turbulence,Kaiser1,Dogic,Beppu} and synthetic microswimmers 
\cite{Sagues}. Last not least, active turbulence was also suggested as a power source for various microfluidic applications 
\cite{Kaiser1,Kaiser2,Kaiser3,Thampi}. 
We refer to those papers and to \cite{zls2016} for
a more detailed description of the physics behind the additional occuring
terms. 

In \cite{zls2016} an analytical approach to (\ref{eqn:min-hyd-mod})
in $L^2(\R^n)$ is presented. 
The aim of this note is prove well-posedness and significant results 
on stability and instability (turbulence)  
in the $\FM(\R^n)$-setting, i.e., in spaces of Fourier transformed Radon
measures. The purpose is to mathematically confirm the asymptotic
behavior observed in simulations and experiments as well as 
the following 'formal' linear stability analysis given in
\cite{Wensink-et-al:Meso-scale-turbulence}:
For $p_0\in\R$ consider the steady state $(0,p_0)$ of (\ref{eqn:min-hyd-mod})
corresponding to a {\em disordered isotropic state}
(see (\ref{disorderedstate})). 
Plugging the wave ansatz
\begin{equation}\label{wavean}
	(v,p):=(0,p_0)+(\eps,\eta)\exp(ik\cdot x+\sigma t),
	\quad k\neq 0,\ x\in\R^n,\ t\ge 0,\ \sigma\in\R, 
\end{equation}
with small $\eps,\eta$
into system (\ref{eqn:min-hyd-mod}) and neglecting the nonlinear terms
yields the characteristic spectral values
\[
	\sigma(k)=-(\alpha+\Gamma_0k^2+\Gamma_2k^4).
\]
Thus unstable (turbulent) modes are expected to exist for
$\Gamma_0 < 0$ and 
    $4\alpha < \Gamma_0^2/\Gamma_2$, or for
    $\Gamma_0 \geq 0$ and $\alpha < 0$.
A similar formal argument leads to stable and unstable modes
for a manifold of ordered polar states (see also the discussion before
Proposition~\ref{orderedstate_lin_stability}).

In \cite{zls2016} 
precise and rigoros results on linear
and nonlinear stability and instability in the $L^2(\R^n)$-setting are given, 
depending on the values of the involved parameters. 
This, however, does not rigorously confirm the formal stability analysis 
above just by the fact that the wave ansatz (\ref{wavean}) is not an
$L^2(\R^n)$-function.
(As it is well known, changing the space, i.e.\ the functional setting, 
in general changes the spectrum,
the growth bound and their relation.) 
On the other hand, it is easy to see that $(v,p)$ as given in
(\ref{wavean}) can be regarded as a Fourier transformed Radon measure,
that is, it belongs to the space $\FM(\R^n)$ (see Remark~\ref{wainfm}). 
In this note we derive precise and rigoros results on linear
and nonlinear stability and instability in the $\FM(\R^n)$-setting 
which justifies the formal argument based on the 
wave modes (\ref{wavean}).

Note that in the context of evolution equations 
the formal stability ana\-ly\-sis given above based on wave modes
of the form (\ref{wavean}) is standard 
in applied literature. The approach in $\FM$-spaces to confirm this
argument in unbounded domains such as $\R^n$, half-spaces or layers 
is developed in \cite{GS13}. It is, e.g., also succesfully applied to confirm
stability of the Ekman spiral for low Reynolds numbers in \cite{gs2012} 
and instability of the Ekman spiral for high Reynolds numbers 
in \cite{fisa2011}.

We organized this note as follows.
In Section~\ref{sec_frm} we briefly recall basic facts on
the space $\FM$. The main part Section~\ref{subsec_wp}
is divided in several subsections. 
In Subsection~\ref{sec_lin} we give precise
information on linear (in-) stability of the steady states depending on
the values of the involved parameters.
In Subsection~\ref{subsec_nlwp} we prove 
well-posedness for the 
generalized Navier-Stokes equations (\ref{eqn:min-hyd-mod}) in 
the $\FM$-setting. 
In fact, we prove existence of a unique maximal strong solution
for arbitrary data and existence of a unique global mild solution 
for small data.
In Subsection~\ref{subsec_stab} we transfer most of the 
results on linear (in-) stability to the nonlinear 
system (\ref{eqn:min-hyd-mod}).

%%%%%%%%%%%%%%%%%%%%%%%%%%%%%%%%%%%%%%%%%%%%%%%%%%%%%%%%%%%%%%%%%%%%%%%%%%%%%
\section{The space of Fourier transformed Radon measures}\label{sec_frm}
%%%%%%%%%%%%%%%%%%%%%%%%%%%%%%%%%%%%%%%%%%%%%%%%%%%%%%%%%%%%%%%%%%%%%%%%%%%%%
%%%%%%%%%%%%%%%%%%%%%%%%%%%%%%%%%%%%%%%%%%%%%%%%%%%%%%%%%%%%%%%%%%%%%%%%%%%%%

We start with basic notation.
For a domain $\Omega\subset \R^n$
and a Banach space $X$
in the sequel $L^p(\Omega,X)$, $1\le p\le\infty$, 
denotes the standard Bochner-Lebesgue space with norm 
\[
	\|u\|_{L^p(X)}=\left(\int_\Omega\|u(x)\|_X^p\,dx\right)^{1/p},
\]
if $1\le p<\infty$ and
$\|u\|_{L^\infty(X)}=\esssup_{x\in\Omega}\|u(x)\|_X$ in case that
$p=\infty$. 
The space of bounded and continuous
functions is denoted $\BC(\Omega,X)$ and we write
$\BUC(\Omega,X)$, if the functions are additionally uniformly continuous. 
As usual, $C^\infty_c(\Omega,X)$ stands for the space of smooth
compactly supported functions.

The symbol $W^{k,p}(\Omega,X)$, $k\in\N_0$, $1\le p\le\infty$, denotes
the standard Sobolev space of $k$-times differentiable functions in
$L^p(\R^n,X)$. Its norm is given as
\[
	\|f\|_{W^{k,p}(X)}:=\biggl(\sum_{|\alpha|\le k}
	\|\partial^\alpha f\|_{L^p(X)}^p\biggr)^{1/p}
\]
with the usual modification if $p=\infty$. 
The class of all bounded and 
linear operators from the space $X$ into the space $Y$
we denote by $\sL(X,Y)$, 
where we write $\sL(X)$ if $X=Y$, and 
$\sigma(A)$ denotes the spectrum of a linear operator
$A:D(A)\subset X\to X$. 

We outline properties of the space of Fourier transformed 
Radon measures $\FM(\R^n)$. For a comprehensive 
and detailed introduction we refer to \cite{GS13}.
%%%%%%%%%%%%%%%%%%%%%%%%%%%%%%%%%%%%%%%%%%%%%%%%%%%%%%%%%%%%%%%%%%%%%%%%%%%%%
\begin{definition}
Let $\mathscr{A}$ be a $\sigma$-Algebra over $\R^n$,
$\mK \in \{\R,\C\}$, and let $\mK^m$ be equipped with the euclidian 
norm $|\cdot|$. A set map $\mu : \mathscr{A} \to \mK^m$ is called a 
finite vector valued Radon measure if
\begin{enumerate}
    \item $\mu$ is a $\mK^m$-valued measure, that is, 
    if $\mu(\emptyset) = 0$ and $\mu$ is $\sigma$-additive;
    \item the variation of $\mu$ defined as
        \begin{equation*}
            \vert \mu \vert (\mathcal{O}) := 
            \sup \left \{ \sum_{E \in \Pi (\mathcal{O})} \vert 
            \mu(E) \vert: \: \Pi(\mathcal{O}) \subseteq 
            \mathscr{A} \: \text{finite decomposition of} \: 
            \mathcal{O} \right \}
        \end{equation*}
    for $\O \in \mathscr{A}$ is a finite Radon measure 
    (that is if $|\mu|(\R^n) < \infty$ and $|\mu|$ is a Borel regular measure).
\end{enumerate}
We denote by $\mathrm{M}(\R^n)=\mathrm{M}(\R^n,\mK^m)$ the space of 
finite vector valued Radon measures.
\end{definition}
%%%%%%%%%%%%%%%%%%%%%%%%%%%%%%%%%%%%%%%%%%%%%%%%%%%%%%%%%%%%%%%%%%%%%%%%%%%%%
From \cite{GS13} we know that $\mathrm{M}(\R^n)$ equipped with the norm 
$\|\mu\|_{\ME(\R^n)} := \|\mu\|_\ME := |\mu|(\R^n)$ is a 
Banach space. Let $\mathscr{B}$ be the Borel $\sigma$-algebra.
Since $\mK^m$ has the Radon-Nikod\'{y}m property
there exists a $\nu_{\mu} \in L^1(\mathbb{R}^n, \vert \mu \vert)$ 
such that we have
$
    \mu(\mathcal{O}) = \int_{\mathcal{O}} \nu_{\mu} 
    \: \text{d}|\mu|$ for $\mathcal{O} \in \mathscr{B}
$.
For $\psi \in \BC(\R^n, \mK^{m\times \ell})$ we set
\begin{equation*}
    \mu \lfloor \psi(\mathcal{O}) := \int_{\mathcal{O}} 
    \psi \, \nu_{\mu} \: \text{d} \vert \mu \vert
    \qquad (\mathcal{O} \in \mathscr{B}),
\end{equation*}
which is well-defined since 
$\mathscr{B} \subset \mathscr{A}$. Elementary properties are listed in
%%%%%%%%%%%%%%%%%%%%%%%%%%%%%%%%%%%%%%%%%%%%%%%%%%%%%%%%%%%%%%%%%%%%%%%%%%%%%
\begin{lemma}\label{fm_lem1}
Let $\mK \in \{\R,\C\}$, $n,m,\ell,j\in\N$,
$\phi \in \BC(\R^n, \mK^{\ell\times j})$, and 
$\psi\in \BC(\R^n, \mK^{m\times \ell})$. Then we have
    \begin{enumerate}
        \item $|\mu \lfloor \psi| \leq |\mu| \lfloor |\psi|$,
        \quad \rm{(2)}\, $\mu \lfloor \psi \in \ME(\R^n)$,
        \quad (3)\, $(\mu \lfloor \psi) \lfloor \phi = \mu \lfloor (\phi \psi)$.
    \end{enumerate}
\end{lemma}
%%%%%%%%%%%%%%%%%%%%%%%%%%%%%%%%%%%%%%%%%%%%%%%%%%%%%%%%%%%%%%%%%%%%%%%%%%%%%
Next, we consider the closed subspace of $\ME(\R^n)$ consisting of measures 
with no point mass at the origin, i.e.,
\begin{equation*}
    \text{M}_{0}(\R^n) := \{ \mu \in \text{M}(\R^n) 
    \, : \, \mu(\{ 0 \}) = 0 \}.
\end{equation*}
We observe that
\[
    L^1(\R^n) \hookrightarrow \ME_0(\R^n) \hookrightarrow \ME(\R^n)
    \hookrightarrow \S'(\R^n).
\]
Hence, the Fourier transform of a Radon measure is defined and given as 
\begin{equation*}
    \hat{\mu} (\xi) = \mu \lfloor \phi_{\xi} (\R^n) \qquad
    \text{with} \qquad
    \phi_{\xi}(x) = (2 \pi)^{-\frac{n}{2}} \, e^{- ix \cdot \xi}.
\end{equation*}
Spaces of Fourier transformed Radon measures then are defined as
\begin{align*}
    \text{FM}(\R^n) &:= \{ \hat{\mu} : \mu \in \text{M}(\R^n) \}, \\
    \text{FM}_0(\R^n) &:= \{ \hat{\mu} : \mu \in \text{M}_0(\R^n) \},
\end{align*}
which are equipped with the norm 
$\|u\|_\FM := \|\cF^{-1}u\|_\ME = \|\cF u\|_\ME$. Both $\FM(\R^n)$ 
and $\FM_0(\R^n)$ are Banach spaces. Furthermore, we define
\[
    \FM^k (\R^n) := \left \{ u \in \FM(\R^n) : 
    \partial^{\alpha}u \in \FM(\R^n) \quad (|\alpha| \leq k) \right \}
\]
for $k \in \N$ and
\[
    \FM^s (\R^n) := \left \{ u \in \FM(\R^n) : (\xi\mapsto \hat{u} 
    \lfloor |\xi|^s) \in \ME(\R^n) \right \}
\]
for $s \ge 0$. The spaces $\FM_0^k(\R^n)$ for 
$k \in \N$ and $\FM_0^s(\R^n)$ for $s \ge 0$ are defined 
accordingly. Note that for $s \in \N$ the two definitions are consistent
thanks to Proposition~\ref{fmfouriermult} below.
From \cite{GS13} we recall the following useful facts.
%%%%%%%%%%%%%%%%%%%%%%%%%%%%%%%%%%%%%%%%%%%%%%%%%%%%%%%%%%%%%%%%%%%%%%%%%%%%%
\begin{lemma}\label{fm_lem2}
Let $u,v \in \FM(\R^n)$. Then we have
    \begin{enumerate}
        \item $\|uv\|_\FM \leq (2 \pi)^{-\frac{n}{2}} \|u\|_\FM \|v\|_\FM,$
        \item $\cF L^1(\R^n) \hookrightarrow \FM_0(\R^n) 
        \hookrightarrow \FM(\R^n) \hookrightarrow \BUC(\R^n)$.
    \end{enumerate}
\end{lemma}
%%%%%%%%%%%%%%%%%%%%%%%%%%%%%%%%%%%%%%%%%%%%%%%%%%%%%%%%%%%%%%%%%%%%%%%%%%%%%
\begin{proposition}\label{fmfouriermult}
For $\sigma \in \emph{BC}(\R^n \backslash \{0 \}, \mK^{m\times \ell})$ we set
$
    Op(\sigma)f := \cF^{-1}(\hat{f} \lfloor \sigma)
$.
Then we have
\begin{align*}
    \|Op(\sigma)\|_{\mathscr L(\FM_0(\R^n,\mK^m),\FM_0(\R^n,\mK^\ell))} 
    &= \|\sigma\|_{L^{\infty}(\R^n
    \backslash \{0 \}, \mK^{m\times \ell})}.
\end{align*}
If $\sigma$ is additionally continuous at the origin then the assertion
holds also with $\FM_0$ replaced by $\FM$.
\end{proposition}
%%%%%%%%%%%%%%%%%%%%%%%%%%%%%%%%%%%%%%%%%%%%%%%%%%%%%%%%%%%%%%%%%%%%%%%%%%%%%

%%%%%%%%%%%%%%%%%%%%%%%%%%%%%%%%%%%%%%%%%%%%%%%%%%%%%%%%%%%%%%%%%%%%%%%%%%%%%
\begin{remark}\label{wainfm}
By the fact that $\cF e^{ik \cdot } = (2 \pi)^{\frac{n}{2}}
\delta (\cdot - k)$ with $\delta$ the Dirac measure, we obtain
$
\|e^{ik \cdot }\|_\FM 
= (2 \pi)^{\frac{n}{2}} \|\delta(\cdot - k)\|_\ME < \infty$.
Hence $e^{ik \cdot } \in \FM_{0}(\R^n)$ for $k\neq 0$ which proves
the wave ansatz (\ref{wavean}) to be a function in $\FM_0(\R^n)$.
\end{remark}
%%%%%%%%%%%%%%%%%%%%%%%%%%%%%%%%%%%%%%%%%%%%%%%%%%%%%%%%%%%%%%%%%%%%%%%%%%%%%

%%%%%%%%%%%%%%%%%%%%%%%%%%%%%%%%%%%%%%%%%%%%%%%%%%%%%%%%%%%%%%%%%%%%%%%%%%%%%
\section{Well-posedness, stability, and turbulence }\label{subsec_wp}
%%%%%%%%%%%%%%%%%%%%%%%%%%%%%%%%%%%%%%%%%%%%%%%%%%%%%%%%%%%%%%%%%%%%%%%%%%%%%
%%%%%%%%%%%%%%%%%%%%%%%%%%%%%%%%%%%%%%%%%%%%%%%%%%%%%%%%%%%%%%%%%%%%%%%%%%%%%

We consider the following physically relevant stationary solutions:
\begin{equation}\label{disorderedstate}
	(v,p)=(0,p_0)
\end{equation}
with a pressure constant $p_0$ and, if $\alpha<0$, additionally
\begin{equation}\label{orderedstate}
	(v,p)=(V,p_0),
\end{equation}
where $V\in B_{\alpha,\beta}
:=\{x\in \mathbb R^n:\ |x|=\sqrt{-\alpha/\beta}\}$, i.e., $V$ 
denotes a constant vector with arbitrary orientation
and fixed swimming speed $|V|=\sqrt{-\alpha/\beta}$.
The steady state (\ref{disorderedstate}) corresponds to a {\em disordered
isotropic state} and (\ref{orderedstate}) to the manifold
$B_{\alpha,\beta}$ of {\em globally
ordered polar states}.

In order to include the steady states, 
as in \cite{zls2016} we consider the following generalized system:
\begin{equation}
	\label{eqn:min-hyd-mod-trans}
	\begin{array}{r@{\,=\,}l}
		\!u_t+\lambda_0\left[(u+V)\cdot\nabla\right] u 
		+(M+\beta|u|^2)u-\Gamma_0\Delta
		u+\Gamma_2\Delta^2u+\nabla q& 
		  f+N(u), \\
	\mbox{div}\,u &   0,\\    
	u(0)&u_0.
	\end{array}
\end{equation}
Here $q=p-\lambda_1 |v|^2$, $M\in\R^{n\times n}$ is a symmetric matrix, and
$N(u)=\sum_{j,k}a_{jk}u^ju^k$ with $(a_{jk})_{j,k=1}^n\subset \R^n$ 
is a quadratic nonlinear term. By setting
\begin{equation}\label{valuesds}
	V=0,\quad M=\alpha,\quad N(u)=0
\end{equation}
we obtain (\ref{eqn:min-hyd-mod}) for $u=v$, i.e., the system corresponding to
the steady state (\ref{disorderedstate}) and by setting
\begin{equation}\label{valuesos}
	V\in B_{\alpha,\beta},\quad M=2\beta VV^T,\quad 
	N(u)=-\beta|u|^2V-2\beta(u\cdot V)u
\end{equation}
we obtain the system for $u=v-V$ corresponding to (\ref{orderedstate}).
Note that for the appearing parameters we always assume that 
\begin{equation}\label{domparameters}
	\lambda_0,\lambda_1,\Gamma_0,\alpha\in\R;
	\qquad \Gamma_2,\beta>0.
\end{equation}
Furthermore, space dimension is always assumed to be $n=2$ or $n=3$.

%%%%%%%%%%%%%%%%%%%%%%%%%%%%%%%%%%%%%%%%%%%%%%%%%%%%%%%%%%%%%%%%%%%%%%%%%%%%%
\subsection{The linearized system}\label{sec_lin}
%%%%%%%%%%%%%%%%%%%%%%%%%%%%%%%%%%%%%%%%%%%%%%%%%%%%%%%%%%%%%%%%%%%%%%%%%%%%%
%%%%%%%%%%%%%%%%%%%%%%%%%%%%%%%%%%%%%%%%%%%%%%%%%%%%%%%%%%%%%%%%%%%%%%%%%%%%%

In this subsection we consider the linearized system
\begin{equation}
        \label{lflin}
        \begin{array}{r@{\ =\ }ll}
                u_t+\lambda_0(V\cdot\nabla) u 
                +Mu-\Gamma_0\Delta
                u+\Gamma_2\Delta^2u+\nabla q& 
                  f&\text{in } (0,\infty)\times\R^n, \\
        \mbox{div}\,u &   0&\text{in } (0,\infty)\times\R^n,\\    
        u(0)&u_0&\text{in }\R^n.
        \end{array}
\end{equation}
In a first step we introduce the Helmholtz projection on 
$\FM_0(\R^n)$. The symbol of the Helmholtz projection 
is defined as $\sigma_P(\xi) := I - \xi\xi^T / |\xi|^2$ and 
the corresponding operator as 
$Pu := \cF^{-1} (\hat{u} \lfloor \sigma_P)$ for $u \in \FM_0(\R^n)$. 
Note that $P$ is bounded on $\FM_0(\R^n)$ by Proposition~\ref{fmfouriermult}. 
We obtain the Helmholtz decomposition
\[
    \FM_0(\R^n) = \FM_{0, \sigma}(\R^n) \oplus G_{\FM}(\R^n),
\]
with
\begin{align*}
    \FM_{0, \sigma}(\R^n) &:= P\FM_0(\R^n) = \{ u \in \FM_0(\R^n)
    : \text{div} \: u = 0 \}, \\
    G_{\FM}(\R^n) &:= \{ \nabla p : p \in \widehat{\FM}_0^1(\R^n) \},
\end{align*}
where $\widehat{\FM}_0^1(\R^n) = \left \{ p \in \D'(\R^n) : \nabla p \in
\FM_0(\R^n) \right \} / \C$, see \cite{GS13}. 
Next, we define the operator associated to (\ref{lflin}) as
\begin{equation}\label{op_alf}
    \begin{aligned}
    &A_{LF}u := \lambda_0 (V \cdot \nabla)u + PMu - \Gamma_0\Delta u +
    \Gamma_2\Delta^2 u, \\
    &D(A_{LF}) := \FM_{0, \sigma}^4(\R^n) 
    := \FM_{0, \sigma}(\R^n) \cap \FM^4(\R^n).
    \end{aligned}
\end{equation}
The Fourier symbol of the operator $A_{LF}$ then reads as
\[
    \sigma_{A_{LF}}(\xi) := \cF^{-1}A_{LF}\cF 
    = \Gamma_2|\xi|^4 + \Gamma_0|\xi|^2 + \sigma_P(\xi)M + i\lambda_0
    V\cdot\xi,
    \quad \xi \in \R^n.
\]
Again thanks to Proposition~\ref{fmfouriermult} 
we obtain
%%%%%%%%%%%%%%%%%%%%%%%%%%%%%%%%%%%%%%%%%%%%%%%%%%%%%%%%%%%%%%%%%%%%%%%%%%%%%
\begin{proposition}\label{op_alf_calculus}
There exists an $\omega > 0$ 
such that $\omega + A_{LF}$ admits a bounded 
$H^\infty$-calculus on $\FM_{0,\sigma}(\R^n)$ with $H^\infty$-angle
$\phi^\infty_{\omega + A_{LF}}<\pi/2$.
\end{proposition}
%%%%%%%%%%%%%%%%%%%%%%%%%%%%%%%%%%%%%%%%%%%%%%%%%%%%%%%%%%%%%%%%%%%%%%%%%%%%%
\begin{proof}
Since $\Gamma_2 > 0$ there exists an $\omega > 0$
and a $\vp_0\in(0,\pi/2)$
such that $\omega+\sigma_{A_{LF}}\in \overline{\Sigma}_{\vp_0}$ and
$|\omega + \sigma_{A_{LF}}| \ge \delta> 0$ on $\R^n\setminus\{0\}$. 
Thus for $\vp\in(\vp_0,\pi/2)$ the symbol
$\xi\mapsto h(\omega + \sigma_{A_{LF}}(\xi))\sigma_P(\xi)$ is bounded and 
continuous on $\R^n\setminus\{0\}$ and satisfies
\[
	\|h(\omega + \sigma_{A_{LF}})\sigma_P\|_{L^\infty(\R^n)} 
	\le C_\vp\|h\|_\infty
	\quad (h\in H^\infty(\Sigma_\vp)),
\]
where $H^\infty(\Sigma_\vp)$ denotes the space of bounded holomorphic
functions on the sector $\Sigma_\vp:=\{z\in\C\setminus\{0\};\ |\arg z|<\vp\}$.
By the fact that
\[
    h(\omega + A_{LF})P = \cF^{-1}h(\omega + \sigma_{A_{LF}})\sigma_P\cF
\]
Proposition~\ref{fmfouriermult} yields
\begin{equation}\label{hinfest}
	\|h(\omega + A_{LF})P\|_\FM \le C_\vp\|h\|_\infty
	\quad (h\in H^\infty(\Sigma_\vp)).
\end{equation}
Setting $h(z):=\lambda(\lambda+z)^{-1}$, estimate (\ref{hinfest})
and the fact that $\omega + A_{LF}$ is invertible imply sectoriality of
$\omega + A_{LF}$ on $\FM_{0,\sigma}(\R^n)$ 
with spectral angle $\phi_{\omega + A_{LF}}<\pi/2$.
Thus the holomorphic functional calculus via the Dunford integral 
is defined as usual, see \cite{DHP03}.
Estimate (\ref{hinfest}) then yields the assertion.
\end{proof}
%%%%%%%%%%%%%%%%%%%%%%%%%%%%%%%%%%%%%%%%%%%%%%%%%%%%%%%%%%%%%%%%%%%%%%%%%%%%%

Note that by the sectoriality of $\omega + A_{LF}$ 
the operator $-A_{LF}$ generates an analytic $C_0$-semigroup 
on $\FM_{0,\sigma}(\R^n)$.
Furthermore, fractional powers 
$(\omega + A_{LF})^\gamma: D((\omega + A_{LF})^\gamma)\to
\FM_{0,\sigma}(\R^n)$, $\gamma>0$, are well-defined, see \cite{DHP03}.
As a consequence of Proposition~\ref{op_alf_calculus} we immediately obtain
%%%%%%%%%%%%%%%%%%%%%%%%%%%%%%%%%%%%%%%%%%%%%%%%%%%%%%%%%%%%%%%%%%%%%%%%%%%%%
\begin{corollary}\label{op_alf_interpolation}
For $\gamma \in(0,1)$ we have
\[
    [\FM_{0,\sigma}(\R^n),D(A_{LF})]_\gamma 
    =D((\omega+A_{LF})^\gamma)  
    = \FM_0^{4\gamma}(\R^n) \cap \FM_{0,\sigma}(\R^n),
\] 
where $[\cdot,\cdot]_\gamma$ denotes the complex interpolation functor.
\end{corollary}
%%%%%%%%%%%%%%%%%%%%%%%%%%%%%%%%%%%%%%%%%%%%%%%%%%%%%%%%%%%%%%%%%%%%%%%%%%%%%
\begin{proof}
By means of Fourier transformation it is straight forward to verify
the second equality, whereas the first equality is a consequence
of \cite[Theorem 1.15.3]{triebel}.
\end{proof}
%%%%%%%%%%%%%%%%%%%%%%%%%%%%%%%%%%%%%%%%%%%%%%%%%%%%%%%%%%%%%%%%%%%%%%%%%%%%%
One advantage of working in $\FM(\R^n)$ 
is reflected by the fact that the operator $\Gamma_2\Delta^2$ 
with domain $D(\Gamma_2\Delta^2) = \FM^4(\R^n)$ 
has $L^1$ maximal regularity.
%%%%%%%%%%%%%%%%%%%%%%%%%%%%%%%%%%%%%%%%%%%%%%%%%%%%%%%%%%%%%%%%%%%%%%%%%%%%%
\begin{proposition}\label{l1_max_reg}
Let $1 \leq p \leq \infty$. For $T(t) := \exp(-\Gamma_2 t \Delta^2)$ and $(\Delta^2 T \star f)(t) := \Delta^2 \int_0^t T(t-s)f(s) \, \emph{d}s$ we have
\begin{enumerate}
    \item $\|\Delta^2 T u_0\|_{L^p(\R_+,\FM(\R^n))} 
    \leq \frac{1}{(\Gamma_2 p)^{1/p}} \|u_0\|_{\FM^{4 - 4/p}}$,
    \item $\|\Delta^2 T \star f\|_{L^1(\R_+, \FM(\R^n))} 
    \leq \frac{1}{\Gamma_2}\|f\|_{L^1(\R_+, \FM(\R^n))}$.
\end{enumerate}
\end{proposition}
%%%%%%%%%%%%%%%%%%%%%%%%%%%%%%%%%%%%%%%%%%%%%%%%%%%%%%%%%%%%%%%%%%%%%%%%%%%%%
\begin{proof}
To prove (1) we have due to Lemma~\ref{fm_lem1}(1) that
\begin{align*}
    \|\Delta^2 e^{- \Gamma_2 t \Delta^2} u_0\|_\FM
    = \|\widehat{u}_0 \lfloor (|\xi|^4 e^{-\Gamma_2 t |\xi|^4})\|_\FM
    \leq \int_{\R^n} |\xi|^4 |e^{- \Gamma_2 t |\xi|^4}| \:
    \text{d}|\widehat{u}_0|(\xi).
\end{align*}
Then the assertion follows since
\begin{align*}
    \|\Delta^2 T u_0\|_{L^p(\R_+, \FM(\R^n))} 
    &\leq \int_{\R^n} |\xi|^4 \|e^{- \Gamma_2 (\cdot) |\xi|^4}\|_{L^p(\R_+)} 
    \: \text{d} |\widehat{u}_0|(\xi)\\
    &\leq \frac{1}{(\Gamma_2 p)^{1/p}} \|u_0\|_{\FM^{4 - 4/p}}.
\end{align*}
Estimate (2) follows from (1) and \cite[Lemma 2.4]{GS11}.
\end{proof}
%%%%%%%%%%%%%%%%%%%%%%%%%%%%%%%%%%%%%%%%%%%%%%%%%%%%%%%%%%%%%%%%%%%%%%%%%%%%%
Consequently, $A_{LF}$ has $L^1$ maximal regularity as well:
%%%%%%%%%%%%%%%%%%%%%%%%%%%%%%%%%%%%%%%%%%%%%%%%%%%%%%%%%%%%%%%%%%%%%%%%%%%%%
\begin{theorem}\label{max_reg}
Let $T\in(0,\infty)$. For $f\in
L^1((0,T),\FM_{0,\sigma}(\R^n))$
and $u_0\in \FM_{0,\sigma}(\R^n)$ there exists a
unique solution $(u,q)$
of (\ref{lflin}) satisfying
\begin{align*}
        &\|u\|_{W^{1,1}((0,T),\FM)}
	        +\|u\|_{L^{1}((0,T),\FM^{4})}
		+\|\nabla q\|_{L^{1}((0,T),\FM)}\\
	&\le C(T)\left(\|f\|_{L^{1}((0,T),\FM)}
	 +\|u_0\|_{\FM}\right)
\end{align*}
with $C(T)>0$ independent of $u,q,f,u_0$.
\end{theorem}
%%%%%%%%%%%%%%%%%%%%%%%%%%%%%%%%%%%%%%%%%%%%%%%%%%%%%%%%%%%%%%%%%%%%%%%%%%%%%
\begin{proof}
By Proposition~\ref{l1_max_reg} the operator $\Gamma_2\Delta^2$ enjoys
$L^1$ maximal regularity also on $\FM_{0,\sigma}(\R^n)$. 
Since the remaining terms in $A_{LF}$ are of
lower order, the assertion follows by a standard perturbation argument.
\end{proof}
%%%%%%%%%%%%%%%%%%%%%%%%%%%%%%%%%%%%%%%%%%%%%%%%%%%%%%%%%%%%%%%%%%%%%%%%%%%%%
Now we consider the spectrum of $A_{LF}$ in order to examine 
stability. For this purpose we set $A_d := A_{LF}$ in case of the 
disordered state (\ref{disorderedstate}). Then the Fourier symbol 
of $A_d$ is given as
\[
    \sigma_{A_d}(\xi) := \Gamma_2 |\xi|^4 + \Gamma_0 |\xi|^2 + \alpha,
    \quad \xi \in \R^n.
\]
If we substitute $s = |\xi|^2$ we can characterize the spectrum 
of $-A_d$ easily by computing the intersection points of $\sigma_{A_d}$. 
We obtain
\begin{equation}\label{intersec_points}
    s_{\pm}^2 = - \frac{\Gamma_0}{\Gamma_2} 
    \left ( \frac{1}{2} \pm \sqrt{\frac{1}{4} - 
    \frac{\alpha \Gamma_2}{\Gamma_0^2}} \right )
\end{equation}
and the following result on (in-)stability:
%%%%%%%%%%%%%%%%%%%%%%%%%%%%%%%%%%%%%%%%%%%%%%%%%%%%%%%%%%%%%%%%%%%%%%%%%%%%%
\begin{proposition}
\label{disorderedstate_lin_stability}
Assume (\ref{domparameters}). Then the $C_0$-semigroup $(\exp(-tA_d))_{t \ge 0}$ 
on $\FM_{0,\sigma}(\R^n)$, which corresponds to the disordered 
isotropic state (\ref{disorderedstate}) is linearly stable 
if $\Gamma_0 < 0$ and $4\alpha > \Gamma_0^2/\Gamma_2$ or 
if $\Gamma_0 \ge 0$ and $\alpha > 0$.
More precisely, it is
\begin{enumerate}
    \item exponentially stable if $\Gamma_0 < 0$ and 
    $4\alpha > \Gamma_0^2/\Gamma_2$ or if 
    $\Gamma_0 \geq 0$ and $\alpha > 0$;
    \item asymtotically stable if  $\Gamma_0 < 0$ and 
    $4\alpha = \Gamma_0^2/\Gamma_2$ or if 
    $\Gamma_0 \geq 0$ and $\alpha = 0$;
    \item exponentially unstable if $\Gamma_0 < 0$ and 
    $4\alpha < \Gamma_0^2/\Gamma_2$ or if 
    $\Gamma_0 \geq 0$ and $\alpha < 0$.
\end{enumerate}
\end{proposition}
%%%%%%%%%%%%%%%%%%%%%%%%%%%%%%%%%%%%%%%%%%%%%%%%%%%%%%%%%%%%%%%%%%%%%%%%%%%%%
\begin{proof}
For the exponential (in-)stability we note that the 
growth bound $\omega((\exp(-tA_d))_{t \ge 0})$ 
and the spectral bound 
$s(-A_d) := \sup\{ \Re \, \lambda : \lambda \in \sigma(-A_d)\}$
coincide, 
since $(\exp(-tA_d))_{t \ge 0}$ is an analytic $C_0-$semigroup, see \cite{nagel}. 
Thanks to (\ref{intersec_points}) relations (1) and (3) are immediate. 
In case of (2) we obtain by Lemma~\ref{fm_lem1}(1) that
\[
    \|\exp(-tA_d)u_0\|_\FM 
    \le |\widehat{u_0}| \lfloor |e^{-t\sigma_{A_d}}|(\R^n)
    = \int_{\R^n} |e^{-t\sigma_{A_d}}| \, \text{d}|\widehat{u_0}|.
\]
Dominated convergence implies $\exp(-tA_d)u_0 \to 0$ for $t \to \infty$ and the assertion is proved.
\end{proof}
%%%%%%%%%%%%%%%%%%%%%%%%%%%%%%%%%%%%%%%%%%%%%%%%%%%%%%%%%%%%%%%%%%%%%%%%%%%%%
Next, we consider the ordered polar state (\ref{orderedstate}). 
We set $A_0 := A_{LF}$ in this case and
\[
    \sigma_{A_0}(\xi) := \Gamma_2 |\xi|^4 
    + \Gamma_0 |\xi|^2 + i \lambda_0 (V \cdot \xi) 
    + 2 \beta \sigma_P(\xi) VV^T,
    \quad \xi \in \R^n.
\]
We note that $\sigma_P(\xi) VV^T$ is a positive semidefinite 
matrix. Thus, zero is an eigenvalue with eigenvector 
$x \in \{V\}^\perp$. Choosing $x, \xi \in \{V\}^\perp$ with $|x| = 1$ 
and $|\xi|$ sufficiently small, we can achieve that
\[
    x^{T} \sigma_{A_0}(\xi) x 
    = \Gamma_2 |\xi|^4 + \Gamma_0 |\xi|^2 < 0,
\]
if $\Gamma_0 < 0$. This proves
%%%%%%%%%%%%%%%%%%%%%%%%%%%%%%%%%%%%%%%%%%%%%%%%%%%%%%%%%%%%%%%%%%%%%%%%%%%%%
\begin{proposition}
\label{orderedstate_lin_stability}
Assume (\ref{domparameters}). Then the $C_0$-semigroup 
$(\exp(-tA_0))_{t \geq 0}$ corresponding to the ordered polar 
state (\ref{orderedstate}) is 
\begin{enumerate}
    \item exponentially unstable on 
    $\FM_{0, \sigma}(\R^n)$ if $\Gamma_0 < 0$;
    \item asymptotically stable on
    $\FM_{0, \sigma}(\R^n)$ if $\Gamma_0 \ge 0$.     
\end{enumerate}
\end{proposition}
%%%%%%%%%%%%%%%%%%%%%%%%%%%%%%%%%%%%%%%%%%%%%%%%%%%%%%%%%%%%%%%%%%%%%%%%%%%%%
\begin{proof}
Assertion (1) is clear due to the discussion above. Assertion (2)
follows completely analogous to the proof of Proposition~\ref{disorderedstate_lin_stability}(2).
\end{proof}
%%%%%%%%%%%%%%%%%%%%%%%%%%%%%%%%%%%%%%%%%%%%%%%%%%%%%%%%%%%%%%%%%%%%%%%%%%%%%
%%%%%%%%%%%%%%%%%%%%%%%%%%%%%%%%%%%%%%%%%%%%%%%%%%%%%%%%%%%%%%%%%%%%%%%%%%%%%
%%%%%%%%%%%%%%%%%%%%%%%%%%%%%%%%%%%%%%%%%%%%%%%%%%%%%%%%%%%%%%%%%%%%%%%%%%%%%
\subsection{Local strong and global mild solvability}\label{subsec_nlwp}
%%%%%%%%%%%%%%%%%%%%%%%%%%%%%%%%%%%%%%%%%%%%%%%%%%%%%%%%%%%%%%%%%%%%%%%%%%%%%
%%%%%%%%%%%%%%%%%%%%%%%%%%%%%%%%%%%%%%%%%%%%%%%%%%%%%%%%%%%%%%%%%%%%%%%%%%%%%
We first construct a maximal solution which includes local wellposedness.
For $T>0$ we define relevant function spaces as
\begin{align*}
  \mE_T &:= W^{1,1}((0,T),\FM_{0,\sigma}(\R^n)) 
  \cap L^1((0,T),\FM^4_0(\R^n)), \\
  {_0\mE}_T &:={_0W}^{1,1}((0,T),\FM_{0,\sigma}(\R^n)) 
  \cap L^1((0,T),\FM^4_0(\R^n)), \\
  \mF_T^1 &:= L^1((0,T),\FM_{0,\sigma}(\R^n)), \quad
  \mF^2 := \FM_{0,\sigma}(\R^n),\\
  \mF_T&:=\mF^1_T\times\mF^2,
\end{align*}
and the linear operator
\[
	L:\mE_T\to \mF_T,\quad Lu:=(\partial_tu+A_{_{LF}}u,u(0)).
\]
Here $u\in{_0W}^{1,1}$ means that $u|_{t=0}=0$.
If we also set
\begin{align}
  H(u) &:= \beta P|u|^2u+\lambda_0 P(u\cdot\nabla)u - PN(u),
\label{defnonlins}\\
	F(u)&:=Lu+(H(u),0),\label{absfulleq}
\end{align}
then the full system (\ref{eqn:min-hyd-mod-trans}) is rephrased as
$F(u)=(f,u_0)$.
%%%%%%%%%%%%%%%%%%%%%%%%%%%%%%%%%%%%%%%%%%%%%%%%%%%%%%%%%%%%%%%%%%%%%%%%%%%%%
\begin{lemma}
We have $H\in C^1(\mE_T,\mF_T^1)$ and its Fr\'echet derivative 
is represented as
\begin{equation}\label{repfrder}
	DH(v)u=P\sum_{|\alpha|\le 1}b_\alpha\partial^\alpha u
	+\lambda_0P(u\cdot\nabla)v,
	\quad u,v\in\mE_T,
\end{equation}
with matrices $b_\alpha=b_\alpha(v)\in
L^\infty((0,T),\FM_0(\R^n,\C^{n\times n}))$.
\end{lemma}
%%%%%%%%%%%%%%%%%%%%%%%%%%%%%%%%%%%%%%%%%%%%%%%%%%%%%%%%%%%%%%%%%%%%%%%%%%%%%
\begin{proof}
First observe that the Sobolev embedding 
\begin{equation}\label{sobemb1}
	W^{1,1}((0,T),X)\hook \BUC((0,T),X) 
\end{equation}
yields
\begin{equation}\label{etinlinf0}
	\mE_T\,\hook\,W^{1,1}((0,T),\FM_0(\R^n))\,\hook\,
	\BUC((0,T),\FM_0(\R^n)).
\end{equation}
Utilizing this and the algebra property of $\FM_0$ we easily obtain
\begin{align*}
	\|(u\cdot\nabla)u\|_{\mF^1_T}
	&\le C\|u\|_{L^\infty(\FM_0)}\|\nabla u\|_{\mF^1_T}\le
	C\|u\|_{\mE_T}^2,\\
	\||u|^2u\|_{\mF^1_T}
	&\le C\|u\|^2_{L^\infty(\FM_0)}\|u\|_{\mF^1_T}\le
	C\|u\|_{\mE_T}^3,\\
	\|N(u)\|_{\mF^1_T}
	&\le C\|u\|_{L^\infty(\FM_0)}\|u\|_{\mF^1_T}\le
	C\|u\|_{\mE_T}^2,
\end{align*}
hence $H:\mE_T\to \mF_T^1$. By the fact that $H$ consists of bi- and trilinear
terms it is obvious that $H\in C^1(\mE_T,\mF_T^1)$ 
(even $H\in C^\infty(\mE_T,\mF^1_T)$). 
The Fr\'echet derivative reads as 
\begin{align*}
  DH(v)u &= \beta P|v|^2 u + 2\beta P (u\cdot v)v
             + \lambda_0 P (u\cdot\nabla)v \\
           & \quad + \lambda_0 P (v\cdot\nabla)u
                    - 2 P \sum_{j,k=1}^n a_{jk} (u^j v^k+u^k v^j).
\end{align*}
From this and (\ref{etinlinf0}) representation (\ref{repfrder})
obviously follows.
\end{proof}
%%%%%%%%%%%%%%%%%%%%%%%%%%%%%%%%%%%%%%%%%%%%%%%%%%%%%%%%%%%%%%%%%%%%%%%%%%%%%
\begin{lemma}\label{freiso}
Let $T\in(0,\infty)$ and fix $v\in \mE_T$. Then we have
\[
	L+(DH(v),0)\in\sLis({\mE}_T,\mF_T).
\]
\end{lemma}
%%%%%%%%%%%%%%%%%%%%%%%%%%%%%%%%%%%%%%%%%%%%%%%%%%%%%%%%%%%%%%%%%%%%%%%%%%%%%
\begin{proof}
By employing representation (\ref{repfrder}) for $B(t):=(DH(v(t)),0)$ we will show
that $B(\cdot)$ is a suitable perturbation of $L$.
In the proof we avoid the use of mixed derivative type theorems, 
since their availability in the underlying situation is not clear.
Therefore we proceed in two steps.

First we will show
that $B_1(t)u:=P\sum_{|\alpha|\le 1}b_\alpha(t)\partial^\alpha u$ 
is relatively bounded by $A_{LF}+\mu$
for $\mu>0$ large enough.
Utilizing the H\"older inequality we can estimate
\begin{align*}
	&\|B_1(t)u\|_{\FM_0}\\
	&\le C\left(\||v(t)|^2\|_{\FM_0}
	      +\|v(t)\|_{\FM_0}\right)\|u\|_{\FM^1_0}\\
	&\le \frac{C}{\mu^{3/4}}\left(\|v\|_{L^\infty((0,T),\FM_0)}^2
	      +\|v\|_{L^\infty((0,T),\FM_0)}\right)
	     \|(\mu+A_{LF}(t))u\|_{\FM_0}
\end{align*}
for all $t\in(0,T)$, $u\in D(A_{_{LF}})$ and $\mu\ge \mu_0$ 
with a certain $\mu_0>0$.
Thus, choosing $\mu$ large enough we can apply
\cite[Theorem~2.5]{saal2004a}
to the result that
\[
	L+(\mu+B_1,0)\in\sLis(\mE_T,\mF_T).
\]
Since $L+(\mu+B_1,0)$ is linear, we can remove the shift $\mu>0$.
(Note that in \cite[Theorem~2.5]{saal2004a} it is assumed that $p>1$.
With the Definition of $L^1$ maximal regularity used here
it is obvious, however, that the Theorem remains true for $p=1$.)

In the second step we show that 
$B_2u:=(\lambda_0 P(u\cdot\nabla)v,0)$ is a lower order 
perturbation of $L+(B_1,0)$.
To this end, we first we consider the case of zero time trace, that is
$u\in{_0\mE_{T}}$.
Observe that then the embedding constant in the 
Sobolev embedding (\ref{sobemb1}) does not depend on the length of
the interval $(0,T)$ if we replace $W^{1,1}$ by its zero trace
version ${_0W}^{1,1}$. As a consequence embedding (\ref{etinlinf0}) is
independent of $T$ too. This yields
\begin{align*}
	\|(u\cdot\nabla)v(t)\|_{L^1((0,T',)\FM_0)}
	&\le C\|\nabla v\|_{L^1((0,T'),\FM_0)}
	\|u\|_{L^\infty((0,T'),\FM_0)}\\
	&\le C\|\nabla v\|_{L^1((0,T'),\FM_0)}
	\|u\|_{{_0\mE_{T'}}}
	\quad (T'\in(0,T)),
\end{align*}
and we obtain
\[
	\|B_2u\|_{\mF_{T'}} 
	\le C \|\nabla v\|_{L^1((0,T'),\FM_0)}
	\|u\|_{{_0\mE}_{T'}}
\]
for all $T'\in(0,T)$ and $u\in{_0\mE}_{T'}$.
Thus, choosing $T'\in(0,T)$ small enough, 
a  standard Neumann series argument implies
\begin{equation}\label{mrzerotime}
	L+B\in\sLis({_0\mE}_{T'},\mF_{T'}).
\end{equation}
Since $L+B$ is linear and $\|v\|_{\mE_T}<\infty$, 
we can iterate this procedure. 
Consequently, (\ref{mrzerotime}) remains true for $T'=T$.
This implies that $L+B$ has maximal regularity on
$\FM_{0,\sigma}$. Thus (\ref{mrzerotime}) remains valid
for general time trace in $\mF^2$.
\end{proof}
%%%%%%%%%%%%%%%%%%%%%%%%%%%%%%%%%%%%%%%%%%%%%%%%%%%%%%%%%%%%%%%%%%%%%%%%%%%%%
Appealing to the local inverse theorem we can now prove the
following result.
%%%%%%%%%%%%%%%%%%%%%%%%%%%%%%%%%%%%%%%%%%%%%%%%%%%%%%%%%%%%%%%%%%%%%%%%%%%%%
\begin{proposition}[Maximal solution]\label{locstrongsol}
Assume (\ref{domparameters}).
For every $u_0\in \FM_{0,\sigma}(\R^n)$ and 
$f\in L^1((0,\infty),\FM_{0,\sigma}(\R^n))$ there exists a 
$T^*>0$ and a unique maximal strong 
solution $(u,p)$ of (\ref{eqn:min-hyd-mod-trans}) such that
\begin{align*}
	u\in \mE_{T},&\quad \nabla p\in L^{1}((0,T),\FM_0(\R^n))
\end{align*}
for all $T\in (0,T^*)$. Either we have $T^*=\infty$ or
the maximal solution satisfies
$\limsup_{t\to T^*}\|u(t)\|_{\FM_0}=\infty$.
\end{proposition}
%%%%%%%%%%%%%%%%%%%%%%%%%%%%%%%%%%%%%%%%%%%%%%%%%%%%%%%%%%%%%%%%%%%%%%%%%%%%%
\begin{proof}
We fix $(f,u_0)\in \mF_T$ and define a reference solution as
\[
	u^*:=L^{-1}(f,u_0)\in \mE_T.
\]
For the Fr\'echet derivative of the nonlinear operator $F\in
C^1(\mE_T,\mF_T)$ 
given in (\ref{absfulleq}) we obtain in view of Lemma~\ref{freiso}
that
\[
	DF(u^*)=L+(DH(u^*),0)\in \sLis(\mE_T,\mF_T).
\]
Utilizing the local inverse theorem, the construction of a unique local
strong solution follows now verbatim the lines of the proof
of \cite[Theorem~1]{zls2016}.

Based on the local well-posedness, as usual, we can show the 
existence of a $T^*>0$ and 
of a unique non-extendible solution $(u,p)$ on $(0,T^*)$.
For the additional property, 
suppose that $\limsup_{t\to T^*}\|u(t)\|_{\FM_0}<\infty$ and nevertheless
$T^*<\infty$. This implies
$u\in \BC([0,T^*),\FM_0)$ thanks to $u\in \mE_T$ for $T<T^*$
and embedding (\ref{etinlinf0}).  
Next, we write
\[
	H(u)(t) 
	= \left(\beta P|u(t)|^2+\lambda_0 P(u(t)\cdot\nabla) -
	Pu(t)^TA\right)u(t)
	=: B(t)u
\]
with $A=(a_{jk})_{j,k=1}^n$.
This allows for regarding (\ref{eqn:min-hyd-mod-trans}) as the 'linear'
system
\[
	(\partial_tu+A_{_{LF}}u+B(\cdot)u,u(0))=(f,u_0).
\]
By the fact that 
\[
	\|B(t)u\|_{\FM_0}\le C\|u(t)\|_{\FM_0^1}\quad (t\in (0,T^*)),
\]
we see that $B(t)$ is a lower order perturbation. It is well-known
that then maximal regularity remains true for $A_{_{LF}}+B(\cdot)$.
In fact, based on a Neumann series argument very similar as, e.g., 
in \cite[Theorem~2.5]{saal2004a} or \cite[Lemma~3]{zls2016}  
it can be proved that
\[
	L+(B(\cdot),0)\in\sLis({\mE}_{T^*},\mF_{T^*}).
\]
By the uniqueness of the solution and due to (\ref{etinlinf0}) 
this gives us 
\[
	u\in\mE_{T^*}\,\hook\, \BUC((0,T^*),\FM_{0,\sigma}).
\]
Thus $\lim_{t\to T^*}\|u(t)\|_{\FM_0}$ exists and starting from the
initial value $u(T^*)$ we can extend the solution $u$ beyond $T^*$
which contradicts its non-extendability.
\end{proof}
%%%%%%%%%%%%%%%%%%%%%%%%%%%%%%%%%%%%%%%%%%%%%%%%%%%%%%%%%%%%%%%%%%%%%%%%%%%%%
In the case of linear exponential stability we obtain existence of
a global mild solution for small data, i.e., a solution of the variation of
constant formula
\begin{equation}\label{vocf}
	u(t)=\exp(-tA_{d})u_0+\int_0^t \exp(-(t-s)A_{d})H(u)(s)ds,
	\quad t>0.
\end{equation}
Besides, the exponential stability transfers to the nonlinear system.
%%%%%%%%%%%%%%%%%%%%%%%%%%%%%%%%%%%%%%%%%%%%%%%%%%%%%%%%%%%%%%%%%%%%%%%
\begin{theorem}\label{mildexpst}
Assume (\ref{domparameters}) such that
$\Gamma_0 < 0$ and $4\alpha > \Gamma_0^2/\Gamma_2$, or such that
$\Gamma_0 \geq 0$ and $\alpha > 0$. 
Then there is a $\kappa>0$ such that, if
        $\|u_0\|_{\FM}<\kappa$,
there exists a unique global mild solution  
$
        u\in\BC([0,\infty),\FM_{0,\sigma}(\R^n))
$
of (\ref{vocf}) satisfying
\[
        \|u(t)\|_{\FM}\le C\e^{-\omega t}\|u_0\|_{\FM}
        \quad (t\ge0)
\]
for some $C,\omega>0$. Furthermore, recovering the pressure via
\[
	\nabla q:=-(I-P)\left[
	\lambda_0\left[u\cdot\nabla\right] u 
		+(M+\beta|u|^2)u-N(u)\right]
		\in L^1((0,T),\FM_0(\R^n))
\]
the pair $(u,\nabla q)$ is the 
unique classical solution of (\ref{eqn:min-hyd-mod-trans}). 
\end{theorem}
%%%%%%%%%%%%%%%%%%%%%%%%%%%%%%%%%%%%%%%%%%%%%%%%%%%%%%%%%%%%%%%%%%%%%%%
\begin{proof}
The proof is very analogous to the proof of 
\cite[Theorem~1.2 and 1.3]{gims2007} and is hence omitted.
\end{proof}
%%%%%%%%%%%%%%%%%%%%%%%%%%%%%%%%%%%%%%%%%%%%%%%%%%%%%%%%%%%%%%%%%%%%%%%

%%%%%%%%%%%%%%%%%%%%%%%%%%%%%%%%%%%%%%%%%%%%%%%%%%%%%%%%%%%%%%%%%%%%%%%%%%%%%
\subsection{Nonlinear turbulence}\label{subsec_stab}
%%%%%%%%%%%%%%%%%%%%%%%%%%%%%%%%%%%%%%%%%%%%%%%%%%%%%%%%%%%%%%%%%%%%%%%%%%%%%
%%%%%%%%%%%%%%%%%%%%%%%%%%%%%%%%%%%%%%%%%%%%%%%%%%%%%%%%%%%%%%%%%%%%%%%%%%%%%

Most of the outcome on linear (in-) stability in the $\FM$-setting 
transfers to the corresponding
nonlinear situation. The transfer of turbulence
follows by principles on linearized instability. Here we apply 
\cite[Corollary 5.1.6]{henry}.
%%%%%%%%%%%%%%%%%%%%%%%%%%%%%%%%%%%%%%%%%%%%%%%%%%%%%%%%%%%%%%%%%%%%%%%%%%%%%
\begin{lemma}\label{nonlinearregularity}
Consider the nonlinearity $H$ given in (\ref{defnonlins}). 
Then we have $H \in C^1(\FM^\eta(\R^n), \FM_{0, \sigma}(\R^n))$ 
for $\eta \ge 1$ and the estimate 
\[
    \| H(u) \|_{\FM} \leq C \| u \|^2_{\FM^\eta} 
    \qquad (\|u\|_{\FM^{\eta}} \leq 1).
\]
\end{lemma}
%%%%%%%%%%%%%%%%%%%%%%%%%%%%%%%%%%%%%%%%%%%%%%%%%%%%%%%%%%%%%%%%%%%%%%%%%%%%%
\begin{proof}
Using the algebra property of $\FM_0(\R^n)$ in 
Lemma~\ref{fm_lem2}(1) we obtain
\begin{align*}
    \| (u \cdot \nabla)u \|_\FM 
    &\leq C \| u \|_\FM \| \nabla u \|_\FM 
    \leq C \|u\|_\FM \|u\|_{\FM^\eta},\\
    \| |u|^2 u \|_\FM &\leq C \| u \|^3_\FM,\\
    \| N(u) \|_\FM &\leq C \| u \|_\FM^2,
\end{align*}
and the claimed estimate follows for $u \in \FM^\eta(\R^n)$ 
with $\|u\|_{\FM^\eta} \le 1$. The estimates also prove
$H \in C^1(\FM^\eta(\R^n), \FM_{0, \sigma}(\R^n))$, 
since $H$ consists of bi- and trilinear terms.
\end{proof}
%%%%%%%%%%%%%%%%%%%%%%%%%%%%%%%%%%%%%%%%%%%%%%%%%%%%%%%%%%%%%%%%%%%%%%%%%%%%%
First, we again examine the (in-) stability of 
the disordered state (\ref{disorderedstate}).
%%%%%%%%%%%%%%%%%%%%%%%%%%%%%%%%%%%%%%%%%%%%%%%%%%%%%%%%%%%%%%%%%%%%%%%%%%%%%
\begin{theorem}\label{disorderedstate_nonlin_stability}
Assume (\ref{domparameters}).
Then the disordered state (\ref{disorderedstate}) 
is nonlinearly
\begin{enumerate}
    \item exponentially stable in 
    $\FM_{0,\sigma}(\R^n)$ 
    if $\Gamma_0 < 0$ and 
    $4\alpha > \Gamma_0^2/\Gamma_2$, or if 
    $\Gamma_0 \geq 0$ and $\alpha > 0$;
    \item unstable in 
    $\FM^{4\gamma}(\R^n) \cap \FM_{0, \sigma}(\R^n)$ for 
    $\gamma \in [1/4,1)$ if $\Gamma_0 < 0$ and 
    $4\alpha < \Gamma_0^2/\Gamma_2$, or if 
    $\Gamma_0 \geq 0$ and $\alpha < 0$.
\end{enumerate}
\end{theorem}
%%%%%%%%%%%%%%%%%%%%%%%%%%%%%%%%%%%%%%%%%%%%%%%%%%%%%%%%%%%%%%%%%%%%%%%%%%%%%
\begin{proof}
(1) is an immediate consequence of Theorem~\ref{mildexpst}.

For (2) first observe that $T^* < \infty$ implies that $u \equiv 0$ 
is unstable, since $\limsup_{t \to \infty} \|u(t)\|_{\FM_0} = \infty$ 
by Proposition~\ref{locstrongsol}. So, w.l.o.g. we can assume $T^* = \infty$.
From Proposition~\ref{disorderedstate_lin_stability} we have 
that $\sigma(-A_d) \cap \{z \in \C : \Re z > 0 \} \neq \emptyset$. 
Thanks to Corollary~\ref{op_alf_interpolation} and
Lemma~\ref{nonlinearregularity} with $\eta=4\gamma\ge 1$ 
we can apply \cite[Corollary 5.1.6]{henry} and the assertion follows.
(In the notation of \cite{henry} we have 
$x_0 = 0$, $A = A_d$, $B = 0$, $f(u)= g(u)= H(u)$, $\alpha=\gamma$, $p=2$.) 
\end{proof}
%%%%%%%%%%%%%%%%%%%%%%%%%%%%%%%%%%%%%%%%%%%%%%%%%%%%%%%%%%%%%%%%%%%%%%%%%%%%%
We obtain a similar result on instability 
of the ordered polar state (\ref{orderedstate}).
%%%%%%%%%%%%%%%%%%%%%%%%%%%%%%%%%%%%%%%%%%%%%%%%%%%%%%%%%%%%%%%%%%%%%%%%%%%%%
\begin{theorem}\label{orderedstate_nonlin_stability}
Let $\Gamma_2 > 0$, $\beta > 0$ and $\Gamma_0, \alpha < 0$. 
Then the ordered polar state (\ref{orderedstate}) is nonlinearly 
unstable in $\FM^{4\gamma}(\R^n) \cap \FM_{0, \sigma}(\R^n)$ for 
$\gamma \in [1/4,1)$.
\end{theorem}
%%%%%%%%%%%%%%%%%%%%%%%%%%%%%%%%%%%%%%%%%%%%%%%%%%%%%%%%%%%%%%%%%%%%%%%%%%%%%
\begin{proof}
Also here the assumptions of \cite[Corollary 5.1.6]{henry} 
are fulfilled thanks to Corollary~\ref{op_alf_interpolation},
Proposition~\ref{orderedstate_lin_stability}(1) and
Lemma~\ref{nonlinearregularity}. 
\end{proof}
%%%%%%%%%%%%%%%%%%%%%%%%%%%%%%%%%%%%%%%%%%%%%%%%%%%%%%%%%%%%%%%%%%%%%%%%%%%%%

%%%%%%%%%%%%%%%%%%%%%%%%%%%%%%%%%%%%%%%%%%%%%%%%%%%%%%%%%%%%%%%%%%%%%%%%%%%%%
\section{Conclusion}
%%%%%%%%%%%%%%%%%%%%%%%%%%%%%%%%%%%%%%%%%%%%%%%%%%%%%%%%%%%%%%%%%%%%%%%%%%%%%
%%%%%%%%%%%%%%%%%%%%%%%%%%%%%%%%%%%%%%%%%%%%%%%%%%%%%%%%%%%%%%%%%%%%%%%%%%%%%

We gave an analytical approach to the active fluid model
proposed by Wensink et al.\ \cite{Wensink-et-al:Meso-scale-turbulence} in 
the $\FM(\R^n)$-setting, i.e., in spaces of Fourier transformed Radon
measures.
In detail we have proved: 
\begin{itemize}
\item[(i)] existence of a unique maximal strong solution for
arbitrary data and existence of a unique global mild (classical) solution
for small data in case of linear exponential stability;
\item[(ii)] results on linear and nonlinear 
stability and instability of the ordered
and the disordered steady states in the $\FM(\R^n)$-setting, 
depending on the values of the occuring physically relevant parameters.
\end{itemize}
By the fact that wave modes belong to $\FM(\R^n)$ (not to $L^2(\R^n)$)
this justifies the typical formal stability analysis based on wave modes
\cite{Wensink-et-al:Meso-scale-turbulence}.
It also justifies mathematically the asymptotic
behavior observed in simulations and experiments
\cite{Wensink-et-al:Meso-scale-turbulence,Kaiser1,Dogic,Doostmohammadi,Beppu,Sagues}, 
in particular
meso-scale turbulence caused by self-propulsion.

\end{document}